\documentclass[]{article}

\usepackage{amsmath}
\usepackage{amssymb}
\usepackage{amsthm}

\usepackage{times}
\usepackage[utf8]{inputenc}
\usepackage[T1]{fontenc}
\usepackage{tikz-cd}
\usepackage{hyperref}
\hypersetup{
	colorlinks,
	citecolor=blue,
	filecolor=black,
	linkcolor=blue,
	urlcolor=magenta
}

\newtheorem{cor}{Corollary}
\newtheorem{lem}{Lemma}
\newtheorem{thm}{Theorem}

\newtheorem{prob}{Question}
\newtheorem{prop}{Proposition}

\theoremstyle{definition}
\newtheorem{defin}{Definition}
\newtheorem*{remark}{Remark}

\newcommand{\forcing}{
	\operatorname{Fn}(S,\mathcal{F},\omega)}

\newcommand{\fors}{\mathbb{P}\ast\dot{\mathbb{Q}}}

\newcommand{\aut}{\operatorname{Aut}}
\newcommand{\pa}{\operatorname{Part}}
\newcommand{\rg}{\operatorname{rg}}
\newcommand{\dom}{\operatorname{dom}}
\newcommand{\age}{\operatorname{Age}}

\newcommand{\K}{\mathcal{K}}

\title{Cohen-like first order structures}
\author{Ziemowit Kostana\footnote{Research of Z. Kostana was supported by the GA\v{C}R project EXPRO 20-31529X and RVO: 67985840.}\\
	Institute of Mathematics
	Czech Academy of Sciences\\
	\v{Z}itná 25, 115 67 Prague, Czech Republic;\\
	University of Warsaw,\\
	Banacha 2, 02-097 Warsaw, Poland \\
	z.kostana@mimuw.edu.pl}

\begin{document}
	
	\maketitle
	
	\begin{abstract}
		We study uncountable structures similar to the Fra\"iss\'e limits. The standard inductive arguments from the Fra\"iss\'e theory are replaced by forcing, so the structures we obtain are highly sensitive to the universe of set theory. In particular, the generic structures we investigate exist only in generic extensions of the universe. We prove that in most of the interesting cases the uncountable generic structures are rigid. Moreover, we provide a (consistent) example of an uncountable, dense set of reals with the group of integers as its automorphism group.
		
	\end{abstract}
	
	{\bf Keywords:} Cohen forcing, Fraisse limit, homogeneous structure, generic structure \\
	{\bf MSC classification:} 06A05 03C25 03C55 03E35

	\section{Generic Structures}
	
	As one looks at the classical construction of a Fra\"iss\'e limit, described for instance in \cite{hod} or \cite{kub}, one might notice that it is much in the spirit of the Baire Theorem. Namely, we show the existence of a universal homogeneous structure by proving that \emph{almost any}, in a suitable sense, countable structure is universal and homogeneous. In fact, universal homogeneous structures form a residual set in certain Polish space. Having that in mind, one might try to construct specific instance of a universal homogeneous structure, mimicking the definition of a Cohen real from the forcing theory. Roughly speaking, a real number is Cohen over some model if it belongs to each residual set from that model. So it is \emph{very generic}, in a sense that for any typical property a real might have, the Cohen real has this property (of course the same can be said about random reals, but with different notion of typicality). This is the idea behind this work. From one side, we want to look at the model theoretic notion of \emph{saturation} as stemming from the forcing language. From the other, we reach to model theory for tools to produce Cohen-like forcing notions (which might often be just different incarnations of the Cohen forcing). \\
	
	One may ask if we can do the similar thing, but replacing Baire category by measure. So is the Fra\"iss\'e limit a \emph{random} structure, in addition to being a \emph{generic} one? This is of course a very vague question, and it is not even clear what a suitable measure space should be. This idea was undertaken by Petrov and Vershik for graphs \cite{pv}, and extended to other structures by Ackerman, Freer, and Patel \cite{afp}. They obtain an elegant internal characterization of Fra\"iss\'e classes for which the Fra\"iss\'e limit is a structure appearing with probability one in certain probability measure space. This happens precisely in the case of Fra\"iss\'e classes in purely relational languages with the Strong Amalgamation Property. The reader is encouraged to consult \cite{afp} for the precise formulation.\\
	
	We assume the reader is familiar with the basics of forcing theory, and model theory. In the first section we develop the language, and prove, or just state, some general properties of the forcings we study. The second section is a short review of the basic notions from the Fra\"iss\'e theory. In the third section we prove that, unlike ordinary Fra\"iss\'e limits, uncountable structures of this kind tend to be rigid. The fourth section is devoted to the construction of an uncountable real order type, with $(\mathbb{Z},+)$ as the group of automorphisms. In the last, fifth section, we collect some open questions, which look relevant for this line of research. Finally, it should be mentioned that this topic used to be informally discussed from time to time already, as kind of folklore idea known to the community. However, up to the author's knowledge, no systematic study of this idea was ever carried out. The closest to it was perhaps a brief, informal note by M. Golshani \cite{golshani}.\\
	As an initial example, look at the following poset.

	$$\mathbb{P}=\{(A,\le)|\, A\in [\kappa]^{<\omega},\text{ A is a linear order} \},$$
	where $\kappa$ is any cardinal, and the ordering is the reversed inclusion. The following subsets are dense, for $\alpha \neq \beta \in \kappa$.
	\begin{itemize}
		\item $D_\alpha=\{(A,\le)|\, \alpha \in A \}$,
		\item $D_{\alpha,\beta}=\{(A,\le)|\, \exists{n< \omega}{\; n \text{ is between }\alpha\text{ and }\beta  }\}$,	
	\end{itemize}
	
	Therefore, for $\kappa=\omega$, the generic filter produces an isomorphic copy of rationals, and for any $\kappa$ it gives some separable $\kappa$-dense order type. We say that a linear order is $\kappa$-dense, if every open interval has cardinality $\kappa$. It is a general phenomenon that for $\kappa=\omega$ this forcing gives the Fra\"iss\'e limit of the given class. An interesting remark, made by M. Golshani in \cite{golshani}, is that every infinite subset of $\omega$ from the ground model is dense in the obtained structure. \\
	
	For this section we adopt the convention that boldface letters $\mathbb{A}$, $\mathbb{B}$ denote first-order structures, while the corresponding capital letters $A$, $B$ denote underlying sets. In further sections we will denote structures and underlying sets with the same letters, as common in mathematics. \\
	
	In the whole paper $\K$ is a class of structures in some countable, relational, first-order language. By $\K_\kappa$ we denote the class of structures from $\K$ of cardinality less than $\kappa$. Relational means in particular that we do not allow constants in our language. We make the following assumptions on $\K$ (see the next section for the definitions):
	\begin{itemize}
		\item $\K$ has the Joint Embedding Property (JEP),
		\item $\K$ has the Amalgamation Property (AP),
		\item $\K$ is hereditary, so if $A\in \K$, and $B\subseteq A$, then $B\in \K$,
		\item $\K$ has infinitely many isomorphism types,
		\item $\K_\kappa$ is closed under increasing unions of length $<\kappa$.
		
	\end{itemize}

	It will prove convenient to introduce a notation paraphrasing the notation for the Cohen forcing in \cite{kunen}. 
	
	\begin{defin}
		Let $\lambda$ be an infinite cardinal number, and $S$ be any infinite set. Denote by $\operatorname{Fn}(S,\K,\lambda)$ the set
		$$\{ \mathbb{A} \in \K|\; A \in [S]^{<\lambda} \},$$
		ordered by the reversed inclusion. 
	\end{defin}
	
	\begin{prop}
		If $\K$ satisfies the SAP, and $\K_\omega$ has at most countably many isomorphism types, then $\operatorname{Fn}(S,\K,\omega)$
		satisfies c.c.c., and even the Knaster condition, for any set $S$.
	\end{prop}
	
	The bound on the number of finite isomorphism types is automatically ensured if $\K$ is a class of structures in a finite language. When the language is countable, it may or may not be true. Finite metric spaces can be viewed as structures in countable language (see the next section), and still there are continuum many pairwise non-isomorphic (non-isometric) 2-element structures. If we restrict to finite metric spaces with rational distances, there are clearly only countably many isomorphism types. The relevance of the SAP is visible in the example discovered by Wiesław Kubi\'s. Let $\mathcal{F}$ be the class of all finite linear graphs, i.e. connected, acyclic, and with degree of every vertex at most 2. It can be easily checked that $\mathcal{F}$ has the AP, but not the SAP. If $S$ is any infinite set, then 
	$\operatorname{Fn}(S,\mathcal{F},\omega)$ forces that $S$ is a linear graph, and each two points of $S$ are in a finite distance. Therefore it collapses $|S|$ to $\omega$.

	\begin{prop} \label{prop1000}
		Let $S$ be any set, and assume $\K_\lambda$  satisfies the SAP. We assume moreover, that for any $\delta<\lambda$ there are at most $\lambda$ many structures from $\K$, with the universe $\delta$. Then 
		$\operatorname{Fn}(S,\K,\lambda)$ is $\lambda$-closed, and if $\lambda^{<\lambda}=\lambda$, then $\operatorname{Fn}(S,\K,\lambda)$ is $\lambda^+$-c.c. 
	\end{prop}
	Notice that we don't count isomorphic types of $\K$-structures of cardinality less than $\lambda$. We take into account the number of different, not only non-isomorphic, ways the ordinal $\delta$ can be endowed with a first-order structure, so that it becomes a member of $\K$. In all but one example, a bound on this number will be guaranteed by the finiteness of the language.
	\begin{proof}[Proof of Proposition \ref{prop1000}]
		To see that $\operatorname{Fn}(S,\K,\lambda)$ is $\lambda$-closed, notice that for a decreasing sequence of conditions $\{p_\alpha|\; \alpha<\delta\}\subseteq \operatorname{Fn}(S,\K,\lambda)$, for $\delta<\lambda$, a lower bound is given by the union $\bigcup\{p_\alpha|\; \alpha<\delta\}$.\\
		To check the $\lambda^+$-c.c. consider a family of conditions $\{\mathbb{A}_\xi|\; \xi<\lambda^+\}$. Using the $\Delta$-system Lemma, we can trim the sequence, so that the sets $\{{A}_\xi|\; \xi<\lambda^+\}$ form a $\Delta$-system with the root $K \in [S]^{<\lambda}$. There are at most $\lambda$-many structures from $\K$ with the universe $K$, so we can assume that for all $\xi\neq \eta<\lambda^+$, we have
		$$\mathbb{A}_\xi \cap \mathbb{A}_\eta = \mathbb{K}.$$
		Now we can use the SAP for the diagram\\
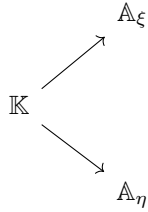
		\begin{center}
			\begin{tikzcd}
				&   \mathbb{A}_\xi 
				&
				& \\
				\mathbb{K} \ar[ur] \ar[dr]
				&
				& 
				\\
				&    \mathbb{A}_\eta
				&
				&
			\end{tikzcd}
		\end{center}
		
		to get a condition stronger from $\mathbb{A}_\xi$ and $\mathbb{A}_\eta$
	\end{proof}
	
	\begin{cor}
		If $\K$ is a class of structures in a finite language, and CH holds, then $\operatorname{Fn}(S,\K,\omega_1)$ is $\omega_2$-c.c.
	\end{cor}
	
	For start we describe structures added by $\operatorname{Fn}(S,\K,\omega)$.
	
	\begin{prop}
		Let $\mathbb{P}=\operatorname{Fn}(\omega,\K,\omega)$, and $G\subseteq \mathbb{P}$ be a generic filter. Then $\bigcup G$ is a structure with the universe $\omega$, isomorphic to the Fra\"iss\'e limit $\mathbb{K}$ of the class $\K_\omega$.
	\end{prop}
	
	\begin{proof}
		In order to ensure that $\bigcup G$ is defined on all $\omega$, we must verify density of the sets
		$$D_n=\{\mathbb{A} \in \mathbb{P}|\; n \in A \},$$
		for $n<\omega$, which is straightforward. To see that we obtain the Fra\"iss\'e limit we must check that each finite extension of a finite substructure is realized. For this purpose, set
		$$E^{i,f}_{\mathbb{B}}=\{\mathbb{A}|\;i:\mathbb{B} \hookrightarrow \mathbb{A} \text{ is an embedding }\implies
		\exists\; g:\mathbb{B}'\hookrightarrow \mathbb{A}\; \text{$g$ is an embedding, and } i=g\circ f \},$$
		where $\mathbb{B}, \mathbb{B}' \in \K$, $f:\mathbb{B} \hookrightarrow \mathbb{B}'$ is an embedding, and $i:\mathbb{B} \hookrightarrow \omega$ is any $1-1$ function.  We also make a technical assumption that both $B$ and $B'$ are disjoint from $\omega$. One could directly apply the AP to show that the sets $E^{i,f}_{\mathbb{B}}$ are dense, however it may be easier to make use of a simple trick, due to W. Kubi\'s.\\
		Fix a structure $\mathbb{A} \in \mathbb{P}$, and assume that $i$, $\mathbb{B}$, $\mathbb{B}'$, $f$ are as above. If $i:\mathbb{B}\hookrightarrow \mathbb{A}$ is not an embedding, then $\mathbb{A} \in E^{i,f}_{\mathbb{B}}$, and we are done. So suppose that $i$ is an embedding. Since $A\subset \omega$, we may extend $\mathbb{A}$ to a structure $\Omega$, isomorphic to $\mathbb{K}$, with the universe $\omega$. Then, since this structure is injective, there exists $g:\mathbb{B}'\rightarrow \Omega$, such that $i=g\circ f$. If we define $\mathbb{A}'=\mathbb{A} \cup g[\mathbb{B}']\subseteq \Omega$, then $\mathbb{A}' \in E^{i,f}_{\mathbb{B}}$. \\
		The universality can either be proved using a similar technology, or we can apply the general fact that in case of relational languages, universality follows from injectivity.
	\end{proof}
	
	Note that we used only countably many dense subsets of $\mathbb{P}$, so the Proposition works under Rasiowa-Sikorski Lemma, without requiring $G$ being "generic" in the sense of the forcing theory.
	
	\section{Review of the Fra\"iss\'e theory}
	
	For the reader's convenience we recall basic notions from the Fra\"iss\'e theory. More detailed introduction can be found in \cite{hod}, or in \cite{kub} in more abstract, category-theoretic setting.
	
	\begin{defin} For a class of structures $\K$ we will say that
		\begin{itemize}
			\item $\K$ has the \emph{Joint Embedding Property} (JEP), if for each $a,b \in \K$ there exists $c \in \K$ such that
			there exist embeddings $a\hookrightarrow c$, and $b \hookrightarrow c$.\\
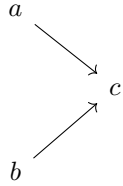
			\begin{center}
				\begin{tikzcd}
					& a  \ar[dr]
					&
					& \\
					& 
					&c
					&
					&
					&
					\\
					& b  \ar[ur]
					&
					& 
				\end{tikzcd}
			\end{center}
			\item $\K$ has the \emph{Amalgamation Property} (AP), if for each pair of embeddings $f:a\hookrightarrow b$, $g:a\hookrightarrow c$, there exists $d\in \K$, together with a pair of embeddings $f':b\hookrightarrow d$, $g':c\hookrightarrow d$, such that $f'\circ f = g'\circ g$.\\
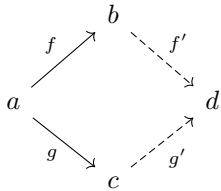
			\begin{center}
				\begin{tikzcd}
					& b  \ar[dr,dashed, "f'"]
					&
					& \\
					a \ar[ur, "f"] \ar[dr, swap, "g"] 
					& 
					&d 
					&
					&
					&
					\\
					& c  \ar[ur, dashed, swap, "g'"]
					&
					& 
				\end{tikzcd}
			\end{center}
			\item $\K$ is \emph{hereditary} if for any $b\in \K$ and any embedding $a \hookrightarrow b$, $a \in \K$.
		\end{itemize}
	\end{defin}
	Notice, that if $\K$ has a weakly initial object, namely a structure which embeds into any element of $\K$, then the JEP follows from the AP. This assumption is typically satisfied, however there are classes with the AP but not the JEP -- for instance the class of all finite fields.
	
	\begin{defin}
		A class $\K$ is a \emph{Fra\"iss\'e class} if is satisfies all properties listed above, and has at most countably many models, up to isomorphism.
	\end{defin}
	
	For checking the Amalgamation Property, we can assume that both initial arrows are identity inclusions. The latter ones however, not always are inclusions, since structures may be "glued together". From time to time we are going to use variants of the AP, which ensures that they aren't.
	
	\begin{defin}
		A class $\K$ has the \emph{Strong Amalgamation Property} (SAP) if for any structures $a,b,c \in \K$ and embeddings $f:a\hookrightarrow b$, $g:a\hookrightarrow c$, there exists $d\in \K$, together with embeddings $f':b\hookrightarrow d$, $g':c\hookrightarrow d$, satisfying $f'\circ f=g' \circ g$, and moreover $\rg{f'}\cap \rg{g'}=\rg{(f'\circ f)}$.
	\end{defin}
	
	The Strong Amalgamation Property essentially means that given any structure $A \in \K$, and two extensions $B_0\supseteq A$, $B_1\supseteq A$, such that $B_0\cap B_1 =A$, we can find bigger $C\in \K$, containing $B_0\cup B_1$ (often $C=B_0\cup B_1$). A close relative of the SAP is the \emph{Splitting Property}. We will say that two embeddings $f:A\hookrightarrow B$ and $g:A\hookrightarrow C$ are \emph{isomorphic}, if there exists an isomorphism $h:B\hookrightarrow C$, such that $h\circ f = g$. The SP is just the SAP for pairs of isomorphic extensions.
	
	\begin{defin} \label{SP}
		A class $\K$ has the \emph{Splitting Property} (SP) if for any structures $a,b,c \in \K$ and isomorphic embeddings $f:a\hookrightarrow b$, $g:a\hookrightarrow c$, there exists $d\in \K$, together with embeddings $f':b\hookrightarrow d$, $g':c\hookrightarrow d$, satisfying $f'\circ f=g' \circ g$, and moreover $\rg{f'}\cap \rg{g'}=\rg{(f'\circ f)}$.
	\end{defin}

	\par  For an infinite structure $A$, we denote by $\operatorname{Age}{A}$ the class of finite substructures of $A$. We will say that $A$ is \emph{locally finite} if each finite subset of $A$ is contained in a finite substructure. This will be the case for example when we are working with purely relational language. 
	
	\begin{defin} A countable structure $A$ is 
		\begin{itemize}
			\item \emph{$\K$-universal}, if for every structure $a \in \K$, there exists an embedding $a \hookrightarrow A$.
			\item \emph{injective}, if for any pair of embeddings $f: a \hookrightarrow A$, $g:a\hookrightarrow b$, where $a,b \in \operatorname{Age}{A}$, there exists an embedding $F:b \hookrightarrow A$, such that $F\circ g=f$.\\
			\begin{center}
				\begin{tikzcd}
					a\arrow[r, "f"]\arrow[dr, swap, "g"]
					& A
					\\& b\arrow[u, dashed,swap, "F"]
				\end{tikzcd}
			\end{center}
			
			\item \emph{homogeneous}, if any isomorphism between finite substructures of $A$ extends to an automorphism of $A$.
		\end{itemize}
	\end{defin}

	\begin{thm}[Fra\"iss\'e, \cite{fraisse}]

		If $\K$ is a Fra\"iss\'e class, then there exists a unique up to isomorphism countable, homogeneous structure $\mathbb{K}$ with $\age{\mathbb{K}}=\K.$
	\end{thm}
	
	The Strong Amalgamation Property for a Fra\"iss\'e class $\K$, with the Fra\"iss\'e limit $\mathbb{K}$, corresponds to a certain property of $\mathbb{K}$. 
	
	\begin{defin} \label{algebraicity}
		The structure $\mathbb{K}$ \emph{has no algebraicity} if for each finite substructure $F\subseteq \mathbb{K}$, and for each $f\in \mathbb{K}\setminus F$, $f$ has infinite orbit under the action of the pointwise stabilizer of $F$ in $\aut\mathbb{K}$. 
	\end{defin}
	
	\begin{thm}[Thm. 7.1.8, \cite{hod}] \label{algebraicity2}
		Let $\K$ be a Fra\"iss\'e class with the Fra\"iss\'e limit $\mathbb{K}$. The following are equivalent.
		\begin{enumerate}
			\item $\K$ has the SAP.
			\item $\mathbb{K}$ has no algebraicity.
		\end{enumerate}
	\end{thm}
	\subsection{Examples} 
	
	Let us review some examples. Typically the only non-trivial condition from the definition of a Fra\"iss\'e class is the AP, so we will briefly describe why it holds for each of the subsequent classes. Verification of other conditions is easy.
	
	\subsubsection{Linear Orders}
	
	\begin{prop} The class of all finite linear orders has the SAP.
	\end{prop}	
	\begin{proof}
		Take a pair of finite linear orders $(K_0,\le_0),(K_1,\le_1)$, such that $\le_0$ and $\le_1$ agree on $L=K_0\cap K_1$. We want to find an ordering $\le_2$ on $K_0 \cup K_1$ extending both $\le_0$ and $\le_1$. This requirement determines $\le_2$ on all pairs, except for ones of the form $\{x_0,x_1\}$, where $x_i\in K_i \setminus L$, for $i=0,1$. We put $x_1<_2x_0$ if there is $y \in L$, such that $x_1<_1y<_0x_0$, and $x_0<_2x_1$ otherwise. It is routine to check that this defines a linear order on $K_0\cup K_1$.
	\end{proof}
	
	It is easy to see that the corresponding Fra\"iss\'e limit is a countable, dense linear order without endpoints. These conditions are satisfied by the ordering of the rationals $(\mathbb{Q},\le)$, and since the Fra\"iss\'e limit is unique, it follows that it is isomorphic to $(\mathbb{Q},\le)$. We have proved the old theorem of Cantor:
	
	\begin{cor}[Cantor, \cite{cantor}]
		Any countable, dense linear order without endpoints is isomorphic to $(\mathbb{Q},\le)$.
	\end{cor}

	\subsubsection{Graphs} \label{graphs}
	In the case of (undirected) graphs, verification of the SAP is straightforward: we just take the set-theoretic union and add no edges.	What is the Fra\"iss\'e limit? Clearly, it is a countably infinite graph $\mathcal{R}$, which satisfies the following axiom:
	\begin{center}
		For each pair of disjoint, finite subsets $A,B \subseteq \mathcal{R}$, there exists a point $x \in \mathcal{R} \setminus(A\cup B)$, connected with every point in $A$, and with no point in $B$. 
	\end{center}
	
	An easy argument by induction shows that this property implies injectivity, so by Lemma 1 it determines $\mathcal{R}$ uniquely, up to isomorphism. 
	
	\par Let $K_n$, $n\ge3$, denote the complete graph on $n$ vertices. We will say that a graph is \emph{$K_n$-free}, if it has no induced subgraph isomorphic to $K_n$. The class of all $K_n$-free graphs is a Fra\"iss\'e class. Let $\mathcal{R}_n$ be the corresponding countable, homogeneous graph. A deep result by Lachlan and Woodrow shows that they essentially exhaust examples of Fra\"iss\'e classes of finite graphs. For a graph $G$, we denote by $G^c$ its complement -- the graph obtained by replacing every edge with non-edge, and the other way around.
	
	\begin{thm}[Lachlan-Woodrow, \cite{lachlan-woodrow}]
		Let $\mathcal{U}$ be a countably infinite, homogeneous graph. Then one of the graphs $\mathcal{U}$ and $\mathcal{U}^c$ is isomorphic to either $\mathcal{R}$, $\mathcal{R}_n$, for $n\ge3$, or a disjoint union of complete graphs of the same size.
	\end{thm}
	
	\subsubsection{Boolean Algebras}
	The class of all finite Boolean algebras\index{Boolean algebra} is a Fra\"iss\'e class. The AP follows from the existence of free products with amalgamation in the category of Boolean algebras, which is described in \cite{bool} Ch. 11. The corresponding homogeneous algebra is the countable, atomless Boolean algebra.

	\subsubsection{Partial Orders}
	The class of all finite partial orders\index{partial order} is a Fra\"iss\'e class with the resulting homogeneous structure known as the \emph{random partial order}.
	
	\begin{prop}
		The class of all partial orders has the SAP.
	\end{prop}
	\begin{proof}
		Fix some partial order $(\mathbb{P},\le)$ and consider two its extensions $(\mathbb{P},\le) \subseteq (\mathbb{P}_0,\le_0), (\mathbb{P}_1,\le_1)$, with $\mathbb{P}=\mathbb{P}_0\cap \mathbb{P}_1$. We define a relation $\le^*$ on $\mathbb{P}_0\cup \mathbb{P}_1$ by the conditions
		$$x_0\le^*x_1 \iff \exists p \in \mathbb{P}\; x_0 \le_0 p \le_1 x_1,$$
		$$x_1\le^*x_0 \iff \exists p \in \mathbb{P}\; x_1 \le_1 p \le_0 x_0.$$
		Verification of transitivity is straightforward, and so is to check that 
		$$\forall x,y \; (x\le^*y \wedge y\le^*x \implies x=y).$$
		Therefore $\le^*$ is a partial ordering of $\mathbb{P}_0 \cup \mathbb{P}_1$.
	\end{proof}
	
	\subsubsection{Groups}
	
	Somewhat more involved Fra\"iss\'e class is the class of finite groups. The amalgamation can be proved using so-called \emph{permutation products} \cite{permprod}. Resulting group is known as the Hall's universal locally finite group, and was first described by Philip Hall in 1959 \cite{hall}.
	
	\par Things are simpler in the case of abelian groups. In this case we can see the AP via reduced products -- for two finite abelian groups $B_0$, $B_1$ with $B_0\cap B_1 =A$ let
	$$E=B_0\times B_1 / \langle(a,-a)|\; a \in A\rangle$$
	
	If we identify $B_0$ and $B_1$ with their natural copies inside $E$, then $E$ witnesses the AP for inclusions
	$A\subseteq B_0$ and $A\subseteq B_1$.	
	
	\begin{prop}
		The group $\mathbb{A}=\displaystyle{\bigoplus_{i<\omega}\mathbb{Q}/\mathbb{Z}}$ is the Fra\"iss\'e limit of the class of all finite abelian groups.
	\end{prop}
	\begin{proof}
		
		First, see that since each finite abelian group is a direct sum of finite cyclic groups, it can be embedded into $\mathbb{A}$. Moreover, each finitely generated subgroup of $\mathbb{A}$ is finite. Why is that? The only way for a finitely generated abelian group to be infinite, is to have an element of an infinite order, but $\mathbb{A}$ has no elements of infinite order. This shows that $\age{\mathbb{A}}$ is exactly the class of finite abelian groups. The group $\mathbb{A}$ is divisible, so it is injective as a $\mathbb{Z}$-module.
		It is tempting to conclude that since $\mathbb{Z}$-modules are just abelian groups, the proof is completed. However, the standard definition of an injective module refers to all group homomorphisms, while our definition of an injective structure takes into account only 1-1 homomorphisms. 
		\par Fix a group monomorphism $f:A_0\hookrightarrow \mathbb{A}$, and a finite group $B \ge A_0$. We want to extend $f$ to $\overline{f}:B\hookrightarrow \mathbb{A}$, keeping it 1-1. We can proceed by induction on the number of generators of $B$, so we can assume that $B$ is generated by the set $A_0\cup \{b\}$, for some $b \in B$. Let $\overline{f}$ be an extension of $f$ obtained from the fact that $\mathbb{A}$ is injective in the algebraic sense. If $\overline{f}$ is 1-1, we are done, so suppose that for some expression $a+b \neq 0$, $\overline{f}(a+b)=0$. By replacing $b$ with $a+b$, we can assume that $\overline{f}(b)=0$. Now notice, that groups $\langle b \rangle$ and $A_0$ have trivial intersection in $B$. Indeed, otherwise for some integer $k$, and $a \in A_0$, we would have $k\cdot b = a$. Now applying $\overline{f}$ both sides, we obtain $\overline{f}(a)=f(a)=0$, and so $a=0$. We may send $b$ to some non-zero element of $\mathbb{A}$, by a homomorphism $g:B\hookrightarrow \mathbb{A}$, which is zero on $A_0$. From the remarks above it is clear that $\overline{f}+g:B\hookrightarrow \mathbb{A}$ is the monomorphism we were looking for.	
	\end{proof}

	\subsubsection{Metric Spaces}
	So far we have been looking only at structures in finite languages. We will call a metric space $(X,d)$ \emph{rational}, if all distances between the points of $X$ are rational numbers. The class of all rational metric spaces is a class of models of a first order theory, in the language consisting of countably many binary relations $d_q$, for all rationals $q>0$, where relation $d_q(x,y)$ is interpreted as  \emph{distance between $x$ and $y$ is at least $q$}. The resulting homogeneous space $(\mathbb{U},d)$ is known as \emph{the rational Urysohn space}, and its completion $\overline{\mathbb{U}}$, as the \emph{Urysohn space}. The space $\overline{\mathbb{U}}$ is uniquely characterized by the following conditions.
	
	\begin{itemize}
		\item $\overline{\mathbb{U}}$ contains an isometric copy of any finite metric space.
		\item Each isometry between between finite subspaces of $\overline{\mathbb{U}}$ extends to a full isometry of $\overline{\mathbb{U}}$ into itself.
	\end{itemize}

	\begin{prop}
		The class of all finite, rational metric spaces has the SAP.
	\end{prop}
	
	\begin{proof}
		Using induction, we can reduce our task to amalgamating two one-point extensions. Fix a finite, rational metric space $(X,d)$, and two extensions $(X_1,d_1)$, $(X_2,d_2)$, where $X_i=X\cup \{x_i\}$, for $i=1,2$, and metrics $d_1$, $d_2$ agree with the metric $d$ on $X$. We want to set the rational distance $q$ between $x_1$ and $x_2$, so that the triangle inequality will hold. This reduces to ensuring that
		$$\forall \; x \in X \; d_1(x,x_1)+d_2(x,x_2) \ge q,$$
		$$\forall \; x \in X\; d_1(x,x_1)+q \ge d_2(x,x_2),$$
		and
		$$\forall \; x \in X\; d_2(x,x_2)+q \ge d_1(x,x_1).$$
		This in turn is just 
		$$\operatorname{dist}(x_1,X\setminus \{x_1\})+\operatorname{dist}(x_2,X\setminus \{x_2\})\ge q\ge |\operatorname{dist}(x_1,X\setminus \{x_1\})-\operatorname{dist}(x_2,X\setminus \{x_2\})|.$$
		Clearly we can find $q>0$ with this property.
	\end{proof}
	
	It makes sense to consider metric spaces with distances restricted to other countable sets. Given any countable subset $D \subseteq [0,\infty)$, let $\mathcal{M}_D$ be the class of finite metric spaces with distances in $D$. While $\mathcal{M}_D$ will always satisfy the SP, it turns out that the AP for $\mathcal{M}_D$ is equivalent to some rather technical condition of $D$, described in \cite{apmetric}. 
	
	\section{Results about rigidity}

	The generic structure added by $\operatorname{Fn}(\omega,\K,\omega)$ is homogeneous, so it can be of some surprise, that forcing on uncountable set gives rise to a rigid structure, at least the typical cases. This is obviously not true if, for example, $\K$ is the class of all finite sets, but it seems to be true in all sufficiently nontrivial cases. This is proved in the first subsection. In the second subsection, we study linear orders added by forcing with countable support, and show that they are not only rigid, but also remain so in any generic extension via a c.c.c. forcing. Note that this is in contrast with the "finite-support-generic" linear orders since, as proved by Baumgartner \cite{baum}, under CH we can add a nontrivial automorphism to any $\omega_1$-dense separable linear order, using a c.c.c. partial order. Recall that a linear order is $\omega_1$-dense, if every open interval has cardinality $\omega_1$.

	\subsection{$\operatorname{Fn}(\omega_1,\K,\omega)$}

	We prove that the  uncountable partial order and the uncountable undirected graph added by the forcing $\operatorname{Fn}(\omega_1,\K,\omega)$ are rigid. Proofs for linear orders, directed graphs, tournaments or finite rational metric spaces are all easy modifications of either of these.
	
	\begin{thm}	
		Let $\mathcal{F}$ be the class of (undirected) graphs, and $S$ be an uncountable set. Then the generic graph added by $\operatorname{Fn}(S,\mathcal{F},\omega)$ is rigid.
	\end{thm}
	\begin{proof} Assume that $p \Vdash "\dot{h}:(S,\dot{E(S)})\rightarrow (S,\dot{E(S)}) \text{ is a non-identity isomorphism}"$. It is easy to check that 
		for every infinite set $F\subseteq S$ from the ground model, and every two different $s,t \in S$, there exists a vertex $e \in F$, with $\{s,e\} \in E(S)$, and $\{t,e\} \notin E(S)$.
		There are clearly uncountably many pairwise disjoint, infinite subsets of $S$ in the ground model, so $h$ must be non-identity on each of them. Therefore there exists an uncountable set 
		$\{p_s|\,s\in S'\subseteq S \}$ of conditions stronger than $p$, with
		$$p_s \Vdash \dot{h}(s)=\overline{s}\neq s.$$
		Without loss of generality we can assume that $\{p_s|\,s\in S'\}$ form a $\Delta$-system with a root $R$, disjoint with $S'$, and the graph structures of all $p_s$ agree on the root. \\
		Fix two different $s,t \in S'$. We can amalgamate $p_s$, and $p_t$ over $R$ in such a way, that $\{s,t\} \in E(S)$, and
		$\{\overline{s},\overline{t}\} \notin E(S)$,
		obtaining some stronger condition $q \in \forcing$. But then $q$ forces, that $\dot{h}$ is not a graph homomorphism.
	\end{proof}

\begin{remark}
	If we were working with tournaments or, more generally, directed graphs, we would have to ensure the corresponding undirected relations:
	$$(s,t) \in E(S),$$
	$$(\overline{s},\overline{t}) \notin E(S).$$
	Then $q$ forces that $\dot{h}$ is not an isomorphism for the exactly same reason. Things are a bit more complicated when we are working with transitive relations, since we need to ensure transitivity in the alamgamation, so we present the full proof for the class of partial orders (the proof for linear orders is obviously reducible to this).
\end{remark}
	
	\begin{thm}	
		Let $\mathcal{F}$ be the class of partial orders, and $S$ be an uncountable set. Then the generic partial order added by $\operatorname{Fn}(S,\mathcal{F},\omega)$ is rigid.
	\end{thm}
	
	\begin{proof}
		Assume that $p \Vdash "\dot{h}:(S,\dot{\leq})\rightarrow (S,\dot{\leq}) \text{ is a non-identity isomorphism}"$. It is easy to check that 
		for every infinite set $E\subseteq S$ from the ground model, $\forcing \Vdash 
		"E \text{ is strongly dense}"$. Strongly dense means that for every $s<t \in S$, there exists $e \in E$, such that $s<e<t$, and for every $s,t \in S$ incomparable, there exists $e_i\in E$, $i=0,1,2,3,4$, with $e_0>s$, $e$ incomparable with $t$, $e_1<s,t; $ $e_2<s$, incomparable with $t$; $e_3>s,t$, and $e_4$ incomparable with both $s$ and $t$. Long story short, each type with parameters (not necessarily from $E$) is realized in $E$.
		There are clearly uncountably many pairwise disjoint, infinite subsets of $S$ in the ground model, and $h$ must be non-identity on each of them. Therefore there exists an uncountable set 
		$\{p_s|\,s\in S'\subseteq S \}$ of conditions stronger than $p$, and
		$$p_s \Vdash \dot{h}(s)=\overline{s}\neq s.$$
		Without loss of generality we can assume that $\{p_s|\,s\in S'\}$ form a $\Delta$-system with a root $R$, disjoint with $S'$, and the order structures of all $p_s$ agree on the root. Suppose also, that for each $s\in S'$, $\overline{s}>s$ (the other cases are handled similarily). Since $S'$ is uncountable, we can further thin it out, so that all embeddings of the form $R\subset R\cup \{s\}$ are pairwise isomorphic, and similarly for $\overline{s}$. Recall that two extensions of a given structure $R$ are isomorphic if there is an isomorphism between them, which is identity on $R$. \\
		Fix two different $s,t \in S'$. There exists an extension
		$R\subset R\cup \{s,t,\overline{s},\overline{t}\}$, with $\{s<t<\overline{t}<\overline{s}\}$. We can amalgamate
		$$p_s \cup \{t<\overline{t} \}$$ and 
		$$p_t \cup \{s<\overline{s} \}$$
		over
		$$R\cup \{s<t<\overline{t}<\overline{s}\},$$
		
		to obtain some condition $q\in \forcing$. But then 
		$q \Vdash s<t$, and $q \Vdash \dot{h}(s)>\dot{h}(t)$, exhibiting a contradiction.
	\end{proof}
	
	It is worth to remark that an uncountable linear order obtained this way satisfies some strong variant of rigidity. Following \cite{as} and \cite{ars}, we say that an uncountable separable linear order $(L,\le)$ is \emph{$k$-entangled}, for some $k \in \mathbb{N}$, if for every tuple $\overline{t} \in \{T,F\}^k$, and any family
	$\{(a_0^\xi,\ldots,a_{k-1}^\xi)|\; \xi <\omega_1 \}$ of pairwise disjoint $k$-tuples from $L$, one can find $\xi \neq \eta <\omega_1$, such that for $i=0,\ldots,k-1 \; a_i^\xi \le a_i^\eta$ iff $\overline{t}(i)=T$. This in particular implies that no two uncountable, disjoint subsets of $L$ are isomorphic. The property of being $k$-entangled for all natural $k$ is featured for example by an uncountable set of Cohen reals, added over some model. Martin's Axiom with negation of CH implies that no uncountable set of reals is $k$-entangled for all $k$ \cite{as}.

	\subsection{$\operatorname{Fn}(\omega_2,\mathcal{LO},\omega_1)$}
	
	We prove, assuming CH, that forcing with countable supports on a set of bigger cardinality gives rise to a rigid linear order, for which we cannot add an automorphism using a c.c.c. forcing. While this result holds under CH,  the c.c.c.-absolute rigidity is clearly preserved by any c.c.c. forcing. In effect, the existence of a rigid $\omega_2$-dense linear order is consistent with any possible value of $2^\omega$, and for example $MA+2^\omega=\kappa$, for any $\kappa=\kappa^{<\kappa}$. Also, we can't replace $\omega_2$ with $\omega_1$ in results of this section. Under CH there exists a unique $\omega_1$-saturated linear order of cardinality $\omega_1$ and as such, it is surely not rigid. But $\operatorname{Fn}(\omega_1,\mathcal{LO},\omega_1)$ forces that the generic order is $\omega_1$-saturated of cardinality $\omega_1$, for the same reasons that $\operatorname{Fn}(\omega,\mathcal{LO},\omega)$ forces the generic order to be $\omega$-saturated (i.e. dense, without endpoints). 
	
	\begin{thm} \label{rigid}
		Let $\mathbb{P}=\operatorname{Fn}(\omega_2,\mathcal{LO},\omega_1)$,
		where $\mathcal{LO}$ denotes the class of all linear orders. Let $(\omega_2,\le)$ be a generic order added by $\mathbb{P}$ over a countable, transitive model $\mathbb{V}$, satisfying CH. Denote by $\mathbb{V}[\le]$ the corresponding generic extension. Let $\mathbb{Q}\in \mathbb{V}[\le]$ be any forcing notion, such that $\mathbb{V}[\le] \models \text{ "$\mathbb{Q}$ is c.c.c." }$, and $H$ be a $\mathbb{Q}$-generic filter in $\mathbb{V}[\le]$. Then the linear order $(\omega_2,\le)$ is rigid in $\mathbb{V}[\le][H]$.
		
	\end{thm}
	
	We will use a simple Lemma assuring that we can amalgamate linear orders in a suitable way.
	
	\begin{lem}
		Let $(L_1,\le_1)$, $(L_2,\le_2)$ be any linear orders, $R=L_1\cap L_2$, \\
		and $(R,\le_1)=(R,\le_2)$. 
		There exists a linear order $\le$ on $L_1\cup L_2$, extending both $\le_1$ and $\le_2$ and satisfying
		$$\forall\; {l_1\in L_1\setminus R}\quad\forall \; {l_2\in L_2\setminus R}\quad {l_1<l_2} \iff \exists\; {r \in R}\;{l_1<_1r<_2l_2}.$$
	\end{lem}

	\begin{proof}
		We take the above formula as the definition.
		
	\end{proof}

	\begin{lem}
		Let $\mathbb{P}=\operatorname{Fn}(\omega_2,\K,\omega_1)$, 
		$\mathbb{P} \Vdash \text{"$\dot{\mathbb{Q}}$ is a c.c.c. forcing notion"},$
		and assume that
		$$\fors \Vdash \overline{h}:\omega_2 \rightarrow \omega_2 \text{ is a bijection}.$$ Then for every $p \in \mathbb{P}$ exists $p_c\le p$ with the property that $(p_c,\dot{\mathbb{Q}}) \Vdash \overline{h}[p_c]=p_c$.
	\end{lem}
	
	\begin{proof}
		Let $\{F_n\}_{n<\omega}$ be a partition of $\omega$ into infinite sets, such that \\
		$\forall{n<\omega}	\; n\le \min{F_n}$. We define a sequence of conditions $p_n \in \mathbb{P}$ by induction, starting with $p_0=p$. Enumerate $p_0=\{r_n|\;
		n \in F_0 \}$. Suppose $p_n$ is defined. Since $n \in F_k$ for some $k\le n$, also $r_n$ is defined. We may take a sequence of $\mathbb{P}$-names with the property
		
		$$p_n \Vdash \text{" $\{\dot{q}_{n+1}^k\}_{k<\omega}$ is a maximal antichain deciding $\overline{h}(r_n)$"}.$$
		
		Since $\mathbb{P}$ is $\sigma$-closed, we can find $p'_n\le p_n$ deciding all the names $\dot{q}_{n+1}^k$ for $k<\omega$. Therefore the set
		$A=\{\beta<\omega_2|\; \exists{k<\omega}\; (p'_n,\dot{q}_{n+1}^k) \Vdash \overline{h}(r_n)=\beta\}$
		is at most countable. Let $p_{n+1}=p'_n\cup A$ (with relations defined arbitrarily), and enumerate $p_{n+1}=\{r_k|\; k \in F_{n+1} \}.$ The inductive step is completed.\\
		Take $p_c=\bigcup_{n<\omega}{p_n}$. We will show that for any $\dot{q}$, with $p_c \Vdash \dot{q} \in \dot{\mathbb{Q}}$, and any $\alpha \in p_c$,
		$(p_c,\dot{q}) \Vdash \overline{h}(\alpha) \in p_c$. Indeed, in this situation there is some $n<\omega$ such that $\alpha \in p_n$. Therefore we can find $k<\omega$ with $\alpha =r_k$, $k\in F_n$. In the $k$-th inductive step we ensure that $(p_{k+1},\dot{q}) \Vdash \overline{h}(r_k) \in p_{k+1}$. It follows that $(p_c,\dot{q})
		\Vdash \overline{h}(r_k)\in p_c$.
	\end{proof}

	\begin{proof}[Proof of Theorem \ref{rigid}]
		Work in $\mathbb{V}$. Let $\dot{\le}$ be a $\mathbb{P}$-name for $\le$. Suppose that
		$\mathbb{P} \Vdash \text{ "$\dot{\mathbb{Q}}$ is c.c.c." }$, and
		$\mathbb{P} \ast \dot{\mathbb{Q}} \Vdash \text{"$\overline{h}:
			(\omega_2,\dot{\le})\rightarrow (\omega_2,\dot{\le)}$ is a non-identity isomorphism"}$. 
		
		\paragraph{Step 0}
		It can be easily verified, that if $\overline{h}$ was identity on a dense set, then it would be identity everywhere. Therefore there exist $\fors$-names $\overline{\delta}_0$, $\overline{\delta}_1$, such that
		
		$$\fors \Vdash \overline{\delta}_0\dot{<}\overline{\delta}_1, \,
		\forall{x \in (\overline{\delta}_0,\overline{\delta}_1)}\; \overline{h}(x)\neq x.$$ 
		
		Fix $(p,\dot{q}) \in \fors$ deciding $\overline{\delta}_0$ and $\overline{\delta}_1$, i.e. 
		$(p,\dot{q}) \Vdash \forall{i\in \{0,1\}}\; \overline{\delta}_i=\delta_i$, for some $\delta_i\in \omega_2$. Without loss of generality, we can assume that $\dot{q}$ is the greatest element of $\dot{\mathbb{Q}}$, so that
		
		$$ (p,\dot{\mathbb{Q}}) \Vdash \delta_0\dot{<}\delta_1, \,
		\forall{x \in (\delta_0,\delta_1)}\; \overline{h}(x)\neq x.$$

		\paragraph{Step 1}
		
		For $\alpha \in \omega_2 \setminus\{\delta_0,\delta_1\}$ we fix a condition $p_\alpha = (p_\alpha, \le_\alpha) \le p$, with 
		$$\delta_0 <_\alpha \alpha <_\alpha \delta_1.$$ Take a sequence of names satisfying
		$$p_\alpha \Vdash \text{"$\{\dot{q}_\alpha^n\}_{n<\omega}$ is a maximal antichain deciding $\overline{h}(\alpha)$"}.$$ Since $\mathbb{P}$ is $\sigma$-closed, we can assume that $p_\alpha$ decides all the names $\dot{q}_\alpha^n$, so the set
		$F(\alpha)=\{\beta<\omega_2|\; \exists{n<\omega}\; (p_\alpha,\dot{q}_\alpha^n) \Vdash \overline{h}(\alpha)=\beta \}$
		is countable. Note, that since $p_\alpha \le p$, $\alpha \notin F(\alpha)$. Finally, we can assume that $F(\alpha) \subseteq p_\alpha$, and, due to Lemma 2, that $(p_\alpha,\dot{\mathbb{Q}}) \Vdash \overline{h}[p_\alpha] = p_\alpha$.
		
		\paragraph{Step 2}
		Using $\Delta$-Lemma for countable sets, we can find $I \subseteq \omega_2$ of cardinality $\omega_2$, with the following
		conditions satisfied 
		\begin{itemize}
			\item $\forall{\alpha \in I}\;\forall{\beta \in I}\; \beta\neq \alpha \implies p_\alpha \cap p_\beta = R$, for some fixed countable $R \subseteq \omega_2$,
			\item $\forall{\alpha\in I}\;\forall{\beta \in I}\; \le_\alpha \restriction{R\times R}=\le_\beta \restriction{R \times R}$,
			\item extensions $R\subset R\cup \{\alpha\}$, for $\alpha \in I$, are pairwise isomorphic,
			\item $\forall{\alpha\in I}\; (p_\alpha,\dot{\mathbb{Q}})
			\Vdash \overline{h}[R]=R$.
		\end{itemize}
		All these conditions, perhaps excluding the last one, are direct consequences of CH. To justify the last claim, notice that $\fors$ is $\omega_2$-c.c. and so the set
		$$A=\{\beta < \omega_2 |\; \exists{(p,\dot{q}) \in \fors}\; \exists{r \in R}\; (p,\dot{q}) \Vdash \overline{h}(r)=\beta\}$$
		
		has cardinality at most $\omega_1$. We choose to $\{p_\alpha|\;
		\alpha \in I \}$ only conditions with $(p_\alpha \setminus R)\cap A=\emptyset$. Take $r \in R$. $(p_\alpha,\dot{\mathbb{Q}})
		\Vdash \overline{h}(r) \in p_\alpha \cap A \subseteq R$.
		
		\paragraph{Step 3}
		Take $\alpha, \beta \in I$, $\alpha \neq \beta$. Using the fact that the extensions $R\subseteq R \cup \{\alpha\}$
		and $R \subseteq R \cup \{\beta\}$ are isomorphic,
		we can extend $\le_\alpha = \le_\beta$ on $R$ to $(R\cup
		\{\alpha, \beta \}, \le_{\alpha,\beta})$ in such a way that there is no element from $R$ between $\alpha$ and $\beta$. We can of course decide that $\alpha <_{\alpha,\beta} \beta$. We now apply Lemma 1 to the pair of isomorphic extensions 
		
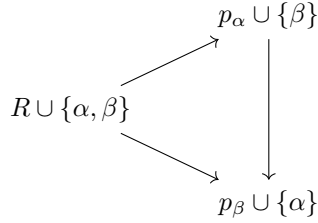
\begin{center}	
			\begin{tikzcd}
				&   p_\alpha \cup \{\beta\} \ar[dd]
				&
				& \\
				R \cup \{\alpha,\beta\} \ar[ur] \ar[dr]
				&
				& 
				\\
				&    p_\beta \cup \{\alpha\}
				&
				&
			\end{tikzcd}
		\end{center}
		where the vertical arrow maps $\beta$ to $\alpha$.\\
		
		Extend $\le_{\alpha,\beta}$ to $p_\alpha \cup p_\beta$, ensuring that
		\begin{itemize}
			
			\item $\neg \exists{r \in R}\; \alpha <_{\alpha,\beta}r<_{\alpha, \beta} \beta$;
			\item $\forall{\gamma \in p_\alpha \setminus (R\cup \{\alpha\})}\; \forall{\eta \in p_\beta \setminus (R\cup\{\beta\})}\\ \gamma <_{\alpha,\beta} \eta
			\iff \exists{r \in R}\;
			\gamma <_{\alpha,\beta} r <_{\alpha,\beta }\eta .$
		\end{itemize}
		
		Take some condition $r\le p_{\alpha,\beta}$ and $\dot{q}$ deciding the values of $\overline{h}(\alpha)$ and $\overline{h}(\beta)$. Then
		$(r,\dot{q}) \Vdash \overline{h}(\alpha) = h(\alpha),\; \overline{h}(\beta)=h(\beta)$. Since there is no element from $R$ between $\alpha$ and $\beta$, and $R$ is $\overline{h}$ invariant, there is also no element from $R$ between $h(\alpha)$ and $h(\beta)$. But since $h(\alpha) \in p_\alpha \setminus \{\alpha\}$, and 
		$h(\beta) \in p_\beta \setminus \{\beta\}$, $h(\beta)<_{\alpha,\beta}h(\alpha)$. Therefore $(r,\dot{q}) \Vdash
		h(\alpha)>h(\beta)$, giving rise to a contradiction. This finishes the proof.\end{proof}
		
		The following argument, suggested by S. Shelah, shows that it is possible to have a separable rigid linear order, whose rigidity is c.c.c.-absolute. By Theorem 24 from \cite{as}, MA($\omega_1$) is consistent with the existence of a rigid set of reals of cardinality $\omega_1$.
		
		\begin{thm}
			Assume MA($\omega_1$), and let $A \subseteq \mathbb{R}$ be a rigid linear order of cardinality $\omega_1$. Then $A$ remains rigid in any generic extension by a c.c.c. forcing.
		\end{thm}
		
		\begin{proof}
			Without loss of generality we may assume that $A=(\omega_1,\le)$. Let $\mathbb{S}$ be any c.c.c. forcing, and suppose towards contradiction that 
			$$\mathbb{S} \Vdash "\dot{f}:(\omega_1,\le)\hookrightarrow(\omega_1,\le) \text{ is a non-identity isomorphism."}$$
			For all $\gamma<\omega_1$, let $A_\gamma \subseteq \mathbb{S}$ be some maximal antichain deciding $\dot{f}(\gamma)$. By Martin's Axiom there is a filter $H\subseteq \mathbb{S}$ intersecting all of the sets $A_\gamma$. Therefore $\dot{f}[H]$ is well defined, and is a non-trivial automorphism of $(\omega_1,\le)$, contrary to the fact that $A$ was rigid.
		\end{proof}

	\section{Linear orders with few automorphisms}
	
	T. Ohkuma proved in \cite{Ohk} that there exist $2^\mathfrak{c}$ pairwise non-isomorphic groups $(G,+) \le(\mathbb{R},+)$, with the property that $\aut{(G,\le)}\simeq (G,+)$, meaning that $G$ has no order-automorphisms other that translations. These groups all have cardinality $\mathfrak{c}$, however the authors of \cite{gghs} have shown that consistently there are uncountable groups of cardinality less than $\mathfrak{c}$ with this property. These are examples of separable, uncountable linear orders, with few, but more than one, automorphisms. We are going to provide one more construction in this spirit.

	\begin{thm} \label{jerusalem}
		It is consistent that there exists an $\omega_1$-dense real order type $(A,\le)$ with a non-identity automorphism $\phi$, such that $\aut{(A,\le)}=\{\phi^k|\; k\in \mathbb{Z} \}$. Moreover, $\phi$ satisfies $\phi(x)>x$ for all $x\in A$.	
	\end{thm}
	
	Let $<_{ord}$ denote the usual order on $\omega_1$. The promised modification of $\operatorname{Fn}(\omega_1,\mathcal{LO},\omega)$ is
	the poset $\mathbb{P}$ consisting of triples $p=(p,\le_p,\phi_p)$ satisfying
	\begin{enumerate}
		\item $\le_p$ is a linear ordering of $p\in [\omega_1]^{<\omega}$,
		\item $\phi_p$ is an increasing bijection between two subsets of $p$,
		\item $\forall{x\in \dom{p}}\; x<_p\phi_p(x)$,
		\item $\forall{x\in \dom{p}}\; \phi(x)<_{ord}x+\omega$, with respect to the ordinal addition on $\omega_1$,
		\item $\forall{x\in \rg{p}}\; \phi^{-1}(x)<_{ord}x+\omega$, with respect to the ordinal addition on $\omega_1$.
	\end{enumerate}
	
	We denote by $(\omega_1,\le)$ the ordering added by $\mathbb{P}$, and by $\phi$ the corresponding automorphism. Before proceeding with the main proof we will see that it is possible to amalgamate finite linear orders together with partial automorphisms in a desired way. It will be convenient to denote by $\pa{(L,\le)}$ the set of finite, partial automorphisms of a linear order $(L,\le)$.
	
	\begin{lem}
		Let $(L_1,\le_1)$, $(L_2,\le_2)$ and $(R,\le_R)=(L_1,\le_1)\cap (L_2,\le_2)$ be finite linear orders. Fix partial automorphisms $\phi_1 \in \pa{(L_1,\le_1)}$, $\phi_2 \in \pa{(L_2,\le_2)}$. We assume that $(L_1,\phi_1)$ and $(L_2,\phi_2)$ are isomorphic extensions of $R$, i.e. there exists an isomorphism $h:(L_1,\le_1) \rightarrow (L_2,\le_2)$ that makes the diagram commutative.
		\begin{center}
			\begin{tikzcd}[column sep=large]
				& (L_1,\le_1) \ar[dd, "h"]  \ar[r, "\phi_1"] & (L_1,\le_1) \ar[dd, "h"]
				&
				& \\
				(R,\le_R) \ar[ur] \ar[dr]
				&
				&
				\\
				& (L_2,\le_2) \ar[r, "\phi_2"] & (L_2,\le_2)
				&
				&
			\end{tikzcd}
		\end{center}
		Take $a,b \in L_1\setminus R$ lying in different orbits of $\phi_1$. There exists a linear order $\le_c$ on $L_1 \cup L_2$ extending $\le_1$ and $\le_2$, and such that $\phi_1\cup \phi_2 \in \pa{(L_1\cup L_2,\le_c)}$, and moreover $a<_c h(a)$, and $h(b)<_cb$.
	\end{lem}
	
	\begin{proof}
		We can assume that $R \subseteq L_1 \subseteq \mathbb{Q}$, and the usual ordering of $(\mathbb{Q},\le)$ extends $\le_1$. We look for an increasing function $f:(L_2,\le_2)\rightarrow (\mathbb{Q},\le)$ such that $f \restriction R = \operatorname{id}_R$, $f[L_2\setminus R] \cap (L_1\setminus R)=\emptyset$, and
		$$f\circ h(a)>a,$$
		$$f\circ h(b)<b.$$
		Indeed, having $f$ as above we will define
		$$x<_cy \iff x<f(y),$$
		for $x \in L_1$ and $y \in L_2$.\\
		
		It can be seen that the only reason why we can't take $f=h^{-1}$ is the disjointness requirement. So we should expect that $f$ will be just a slight distortion of $h^{-1}$. We must also ensure that $\phi_1 \cup \phi_2$ will be order-preserving.\\
		Let $\{x_1,\ldots,x_n \}$ be a $\le_1$-increasing enumeration of $L_1$. For $k=1,\ldots,n$ choose an open interval $I_k$ around $x_k$ in such a way that all intervals obtained this way are pairwise disjoint, and for $l\neq k=1,\ldots,n$ if $x_l=\phi_1^m(x_k)$, then $ \overline{\phi}_1^m[I_l]=I_k$, where $\overline{\phi}_1:(\mathbb{Q},\le)\rightarrow (\mathbb{Q},\le)$ is an extension of $\phi_1$.\\
		For each $k$ we choose $f(h(x_k))\in I_k \setminus \{x_k \}$, so that $\phi_1^m(f\circ h(x_k))=f\circ h \circ \phi_1^m(x_k)$, for $m\in \mathbb{Z}$, whenever this expression makes sense. We also ensure inequalities $f \circ h(a)>a$ and $f\circ h(b)<b$.
	\end{proof}
	
	\begin{prop}
		$\mathbb{P}$ satisfies the Knaster condition.
	\end{prop}
	
	\begin{proof}
		Let $\{p_\alpha=(p_\alpha,\le_\alpha,\phi_\alpha)|\; \alpha<\omega_1 \} \subseteq \mathbb{P}$. We choose a $\Delta$-system
		$\{p_\alpha|\; \alpha\in S \}$, with some additional properties:
		\begin{itemize}
			\item $\forall{\alpha \in S}\; \forall{\beta \in S}\; \alpha \neq \beta \implies (p_\alpha,\le_\alpha)\cap (p_\beta,\le_\beta)=(R,\le_R)$, for some fixed ordering $\le_R$ of $R$,
			\item $\phi_\alpha[R]\subseteq R$,
			\item $\phi^{-1}_\alpha[R]\subseteq R$.
		\end{itemize}
		
		For ensuring the last two properties we use 4. and 5. from the definition of $\mathbb{P}$. To obtain an uncountable set of pairwise comparable conditions, we now only have to trim $\{p_\alpha|\; \alpha \in S \}$, so that $\phi_\alpha \restriction R$ does not depend on $\alpha$, and this is clearly possible. 
	\end{proof}

	\begin{lem}
		For every $\alpha_0 \in \omega_1$, the orbit of $\alpha_0$ under $\phi$ is cofinal and coinitial in $(\omega_1,\le)$
	\end{lem}
	\begin{proof}
		It is easy to see that the required family of dense sets is 
		$$E_\beta=\{p=(p,\le_p,\phi_p) \in \mathbb{P}|\; \{\alpha_0,\beta\} \subseteq p, \; \exists{k\ge 0}\; \beta<_p \phi_p^k(\alpha_0),\; \phi_p^{-k}(\alpha_0)<_p \beta  \},$$	
		for $\beta \in \omega_1$.
		
		In order to check that $E_\beta$ is dense, fix some condition $p=(p,\le_p,\phi_p) \in \mathbb{P}$ and $\beta < \omega_1$. We can assume that $\{\alpha_0,\beta\} \subseteq p$. In order to extend $p$ so that it belongs to $E_\beta$, we embed $(p,\le_p)$ into the set of algebraic numbers $A$. Now we can extend $\phi_p$ to an increasing function $\overline{\phi}:A \rightarrow A$, such that for some rational $\epsilon>0$ $\forall{a \in A}\; \overline{\phi}(a)>a+\epsilon$. It is clear that the orbit of $\alpha_0$ under $\overline{\phi}$ is both cofinal and coinitial in $A$. Finally we just cut out a suitable finite fragment of $\overline{\phi}$, and extend $p$ accordingly.
	\end{proof}
	
	\begin{lem}
		For each isomorphism $h:(\omega_1,\le)\rightarrow (\omega_1,\le)$, and for every uncountable set $F \subseteq \omega_1$, there exist $\alpha \in F$ and $k \in \mathbb{Z}$, such that $h(\alpha)=\phi^k(\alpha)$.
	\end{lem}
	
	\begin{proof}
		Fix a sequence of names for elements of $F$, $\{\dot{x}_\alpha|\; \alpha <\omega_1 \}$. Let $$p \Vdash \dot{h}:(\omega_1,\dot{\le}) \rightarrow (\omega_1,\dot{\le}) \text{ is an isomorphism}.$$ For every $\alpha <\omega_1$ we fix a condition $p_\alpha=(p_\alpha,\le_\alpha,\phi_\alpha) \le p$, so that
		$p_\alpha \Vdash \dot{x}_\alpha=x_\alpha, \;  \dot{h}(x_\alpha)= \overline{x}_\alpha$, for some ordinals $x_\alpha, \overline{x}_\alpha \in \omega_1$. We can also assume that for each $\alpha$, $x_\alpha \neq \overline{x}_\alpha$, for otherwise we just take $k=0$. \\
		We choose an uncountable $\Delta$-system $\{ p_\alpha|\; \alpha \in S \}$, and make it as uniform as possible:
		
		\begin{itemize}
			\item $\forall{\alpha \in S}\; \forall{\beta \in S}\; \alpha \neq \beta \implies (p_\alpha,\le_\alpha)\cap (p_\beta,\le_\beta)=(R,\le_R)$, for some fixed ordering $\le_R$ of $R$,
			\item $\phi_\alpha[R]\subseteq R$,
			\item $\phi^{-1}_\alpha[R]\subseteq R$,
			\item extensions $(R,\le_R) \subseteq (R\cup \{x_\alpha \},\le_\alpha)$ are pairwise isomorphic,
			\item extensions $(R,\le_R) \subseteq (R\cup \{ \overline{x}_\alpha \},\le_\alpha)$ are pairwise isomorphic,
			\item extensions $(R,\le_R) \subseteq (p_\alpha,\le_\alpha)$ are pairwise isomorphic,
			\item The way $\phi_\alpha$ acts on $p_\alpha$ is independent from the choice of $\alpha \in S$. More precisely, $\forall{\alpha \in S} \; \forall{\beta \in S}$ the following diagram commutes
			
			\begin{center}
				$$
				\begin{tikzcd}
					p_\alpha \arrow{r}{\phi_\alpha} \arrow{d}{h} & p_\alpha \arrow{d}{h}\\
					p_\beta \arrow{r}{\phi_\beta} & p_\beta
				\end{tikzcd}
				$$
			\end{center}
			where $h$ is the unique isomorphism between $(p_\alpha,\le_\alpha)$ and $(p_\beta.\le_\beta)$. 
		\end{itemize}
		In particular, the unique isomorphism $h$ maps $x_\alpha$ to $x_\beta$, and $\overline{x}_\alpha$ to $\overline{x}_\beta$. Fix $\alpha \in S$. We claim that $x_\alpha$ and $\overline{x}_\alpha$ are in the same orbit of $\phi_\alpha$. For otherwise, we fix $\beta \in S \setminus \{\alpha \}$, and apply Lemma 3 for $a=x_\alpha$ and $b=\overline{x}_\alpha$. This way we obtain a condition $$q=(p_\alpha \cup p_\beta,\le_q,\phi_\alpha \cup \phi_\beta) \le p_\alpha, p_\beta,$$ satisfying $x_\alpha <_q x_\beta$, and $\overline{x}_\beta<_q \overline{x}_\alpha$. But then $$q \Vdash \dot{x}_\alpha \dot{<} \dot{x}_ \beta, \; \dot{h}(\dot{x}_\beta) \dot{<} \dot{h}(\dot{x}_\alpha),$$ contrary to the choice of $p$. In conclusion 
		$p_\alpha \Vdash \exists{k\in \mathbb{Z}} \; \dot{h}(\dot{x}_\alpha)= \dot{\phi}_\alpha^k(\dot{x}_\alpha)$.
	\end{proof}
	
	Now we are in position to prove Theorem \ref{jerusalem}. 
	
	\begin{proof}
		
		Since $(\omega_1,\le)$ is separable, we can replace it by an isomorphic copy $A \subseteq \mathbb{R}$, $A$ being $\omega_1$-dense. Then $\phi:A \rightarrow A$ is an increasing bijection, strictly above the diagonal (i.e. $\forall x \; x<\phi(x)$). Let $h:A \rightarrow A$ be any increasing bijection. Both $h$ and $\phi$ extend uniquely to the whole real line, so we can assume that $\phi, h :\mathbb{R} \rightarrow \mathbb{R}$ are continuous, increasing bijections. 
		\par For $k \in \mathbb{Z}$, put $F_k=\{ x\in \mathbb{R}|\; h(x)=\phi^k(x) \}$. By continuity, sets $F_k$ are closed, and by Lemma 5, $\bigcup_{i \in \mathbb{Z}} F_i$ is dense. Fix some $k \in \mathbb{Z}$ for which the set $F_k$ is nonempty. We aim to prove that $F_k=\mathbb{R}$. If not, there exists $x \in F_k$, and $\delta>0$ satisfying at least one of conditions
		$$(x,x+\delta) \cap F_k = \emptyset,$$ and
		$$(x-\delta, x) \cap F_k = \emptyset.$$
		
		Let us assume the first case, the other being similar. Since the union of the sets $F_i$ is dense, we can find a decreasing sequence $\{x_n \}_{n<\omega}$, converging to $x$, and integers $k_n$, for which $h(x_n)=\phi^{k_n}(x_n)$. \\
		
		Suppose that for infinitely many $n$, the inequality $k_n>k$ holds. By replacing $\{x_n\}_{n<\omega}$ with a subsequence, we may assume that this is the case for all $n<\omega$. Then 
		$$\phi^{k_n}(x_n)\ge \phi^{k+1}(x_n)\underset{n \rightarrow \infty}{\longrightarrow} \phi^{k+1}(x)>\phi^k(x)=h(x),$$
		which contradicts $\underset{n \rightarrow \infty }{\lim} \phi^{k_n}(x_n)=h(x)$. \\
		If for infinitely many $n$ the inequality $k_n<k$ holds, we proceed in analogous manner. The only way out is that $k_n=k$ for all but finitely many $n$, but this in turn contradicts $(x,x+\delta)\cap F_k = \emptyset$. Therefore $F_k = \mathbb{R}$, and the theorem is proved.
	\end{proof}

	\section{Problems}
	
	\begin{prob}
		Are there some natural conditions for a class $\mathcal{F}$, ensuring that $\forcing$ is, up to completion, the same as the Cohen forcing? 
	\end{prob}

	\begin{prob}
		The automorphism group of an $\omega_1$-dense real order type can be very big or trivial. Are some intermediate options possible? By Theorem 5 it can be isomorphic to $(\mathbb{Z},+)$. Can it be isomorphic, for example, to $(\mathbb{Q},+)$?
	\end{prob}


\begin{thebibliography}{HD}
		\baselineskip=15pt
		
		\begin{small}
			
			\bibitem[1]{afp} N. Ackerman, C. Freer, R. Patel, \textit{Invariant measures concentrated on countable structures}, Forum of Mathematics, Sigma. Vol. 4. Cambridge University Press, 2016.
			
			\bibitem[2]{as} U. Avraham, S. Shelah, \textit{Martin's Axiom does not imply that every two $\aleph_1$-dense sets of reals are isomorphic}, Israel Journal of Mathematics, vol. 38, Nos. 1-2, 1981
			
			\bibitem[3]{ars} U. Avraham, M. Rubin, S. Shelah, \textit{On the consistency of some partition theorems for continuous colorings, and the structure of $\aleph_1$-dense real order types}, Annals of Pure and Applied Logic 29 (1985), 123-206
			
			\bibitem[4]{baum} J.E. Baumgartner, \textit{All $\aleph_1$-dense sets of reals can be isomorphic}, Fund. Math. 79 (1973), 101-106
			
			\bibitem[5]{cantor}G. Cantor, \textit{Beitr\"age zur Begr\"undung der transfiniten Mengenlehre}, Mathematische Annalen, 46 (4), 481-512
			
			\bibitem[6]{apmetric} C.Delhomm\'e, C. Laflamme, M. Pouzet, N. Sauer, \textit{Divisibility of countable metric spaces}, European Journal of Combinatorics 28, no. 6 (2007), 1746-1769.
			
			\bibitem[7]{fraisse}R. Fra\"iss\'e, \textit{Sur quelques classifications des syst\'emes de relations}, Publ. Sci. Univ. Alger. S\'er. A. 1 (1954) 
			
			\bibitem[8]{gghs} A.M.W. Glass, Y. Gurevich, W.C. Holland, S. Shelah, \textit{Rigid homogeneous chains}, Math. Proc. Camb. Phil. Soc. (1981) 89, 7	
			
			\bibitem[9]{golshani} M. Golshani, \textit{Fra\"iss\'e limits via forcing}, preprint, \href{https://arxiv.org/abs/1902.06206}{arXiv:1902.06206} (2019)
			
			\bibitem[10]{hall} P. Hall, \textit{Some Constructions for Locally Finite Groups}, J. London Math. Soc,
			34 (1959), 305-19.
			
			
			\bibitem[11]{hod} W. Hodges, \textit{Model theory, Encyclopedia of Mathematics and its Applications 42.}, Cambridge University Press, Cambridge, 1993
			
			
			\bibitem[12]{bool} S. Koppelberg, edited by D. Monk with R. Bonnet, \textit{Handbook of Boolean algebras, vol. 1}, North-Hollad, 1989
			
			
			\bibitem[13]{kub} W. Kubi\'s, \textit{Fra\"iss\'e sequences: category-theoretic approach to universal
				homogeneous structures}, Annals of Pure and Applied Logic 165(11), 2014 1755-1811
			
			\bibitem[14]{kunen} K. Kunen, \textit{Set Theory. Introduction to Independence Proofs}, 1980 Elsevier Science B.V.
			
			\bibitem[15]{lachlan-woodrow} A. Lachlan, R. E. Woodrow, \textit{Countable ultrahomogeneous undirected graphs}, Transactions of the American Mathematical Society (1980), 51-94.
			
			\bibitem[16]{permprod} B.H. Neumann, \textit{Permutational Products of Groups}, J. Austral. Math. Soc. Ser. A (1959), 299-310. 
			
			\bibitem[17]{Ohk} T. Ohkuma, \textit{Sur quelques ensembles ordonn\'es lin\'eairement}, Fund. Math. 43 (1955), 326-377
			
			\bibitem[18]{pv} F. Petrov, A. Vershik, \textit{Uncountable graphs and invariant measures on the set of universal countable graphs}, Random Structures and Algorithms 37.3 (2010): 389-406.
			
			
		\end{small}
	\end{thebibliography}
\end{document}